\documentclass{article}
\usepackage{latexsym,amsfonts,amsthm,amsmath,amscd,amssymb}
\usepackage[dvips]{graphicx}

\newtheorem{lemma}{Lemma}[section]
\newtheorem{theorem}[lemma]{Theorem}
\newtheorem{remark}[lemma]{Remark}
\newtheorem{proposition}[lemma]{Proposition}
\newtheorem{corollary}[lemma]{Corollary}

\newtheorem{example}[lemma]{Example}
\newtheorem{definition}[lemma]{Definition}

\newcommand\matR{{\mathbb{R}}}

\begin{document}

\title{On the notion of scalar product for finite-dimensional diffeological vector spaces}

\author{Ekaterina~{\textsc Pervova}}

\maketitle

\begin{abstract}
\noindent It is known that the only finite-dimensional diffeological vector space that admits a diffeologically smooth scalar product is the standard space of appropriate dimension. In this note we consider a way 
to circumnavigate this issue, by introducing a notion of pseudo-metric, which, said informally, is the least-degenerate symmetric bilinear form on a given space. We apply this notion to make some observation 
on subspaces which split off as smooth direct summands (providing examples which illustrate that not all subspaces do), and then to show that the diffeological dual of a finite-dimensional diffeological vector 
space always has the standard diffeology and in particular, any pseudo-metric on the initial space induces, in the obvious way, a smooth scalar product on the dual.

\end{abstract}

\section*{Introduction}

One of the first surprising findings which one makes when encountering diffeology for the first time is the inequality $L^{\infty}(V,W)<L(V,W)$, which says that the space of all (diffeologically) smooth linear maps 
between two diffeological vector spaces, $V$ and $W$, can be strictly smaller than that of all maps that are simply linear.\footnote{Say what? It happens in functional analysis already? Well, whatever. I prefer 
to keep it simple (jokes apart, in diffeology it happens in finite dimension already).} Stemming from that, an even more surprising situation presents itself: a (finite-dimensional) diffeological vector space that 
does not admit a smooth scalar product. This is a known fact (see \cite{iglesiasBook}, p. 74, Ex. 70), and easily established at that, yet, it is still surprising, in and of itself, and also for how easily this can be 
illustrated, using the presence of just one non-differentiable (in the usual sense) plot.

Thus, as mentioned in the above-given reference, any finite-dimensional diffeological vector space that admits a smooth scalar product is necessarily the usual $\matR^n$, with its usual smooth structure (the 
diffeology that consists of all usual smooth maps). The choices at this point are, to abandon the whole affair (meaning to concentrate on infinite-dimensional spaces, where the similar situation does not seem 
to occur), to consider a kind of pseudo-metrics (meaning the sort of least degenerate symmetric bilinear form that exists on a given space), or, finally, to re-define scalar product as ones taking values in $\matR$ 
endowed, not with the standard diffeology, but with the piecewise-smooth diffeology. We concentrate on the second of the above approaches, and it does lead to a few interesting conclusions, mainly regarding 
the diffeological dual and the fact that a pseudo-metric induces an isomorphism of it with a specific subspace, which is a smooth summand. Furthermore, it gives rise to a smooth scalar product on the dual, 
reflecting as much as possible of the usual duality for the standard vector spaces.

Finally, a disclaimer: a lot of what is written here might be of a kind of implicit knowledge for people working in the area. Part of the aim of this paper is to collect these facts in one place, and to make explicit 
what is implicit elsewhere.\footnote{Putting an interesting statement as an exercise in a book is just an invitation for somebody to re-discover it in good faith. Just sayin'.}

\paragraph{Acknowledgments} Some of the best things in life happen by chance: a chance encounter, a chance phrase, a chance look... This sounds trivial, but what is surprising is that I found this true of 
professional life as well; some of my most productive moments came about by chance. The very existence of this work (and those that followed it) is due to chance, and for it, I would like to thank the people who,
incidentally, are also responsible for some of the best moments in the last couple of years (all of them, of course, due to chance). These people are Prof. Riccardo Zucchi and Prof. Patrick Iglesias-Zemmour. I 
also would like to thank all the participants of the workshop ``On Diffeology etc.'' (Aix en Provence, June 24-26, 2015) where I had my first opportunity to discuss the reasoning that is at the origin of this paper.

\section{Diffeology and diffeological vector spaces}

We briefly recall the main definitions regarding diffeological spaces and diffeological vector spaces.

\paragraph{Diffeological spaces and smooth maps} We first recall the notion of a diffeological space and that of a smooth map between such spaces.

\begin{definition} \emph{(\cite{So2})} A \textbf{diffeological space} is a pair $(X,\mathcal{D}_X)$ where $X$ is a set and $\mathcal{D}_X$ is a specified collection of maps $U\to X$ (called \textbf{plots}) for
each open set $U$ in $\matR^n$ and for each $n\in\mathbb{N}$, such that for all open subsets $U\subseteq\matR^n$ and $V\subseteq\matR^m$ the following three conditions are satisfied:
\begin{enumerate}
\item (The covering condition) Every constant map $U\to X$ is a plot;
\item (The smooth compatibility condition) If $U\to X$ is a plot and $V\to U$ is a smooth map (in the usual sense) then the composition $V\to U\to X$ is also a plot;
\item (The sheaf condition) If $U=\cup_iU_i$ is an open cover and $U\to X$ is a set map such that each restriction $U_i\to X$ is a plot then the entire map $U\to X$ is a plot as well.
\end{enumerate}
\end{definition}

Usually, instead of $(X,\mathcal{D}_X)$ one writes simply $X$ to denote a diffeological space.

\begin{definition} \emph{(\cite{So2})} Let $X$ and $Y$ be two diffeological spaces, and let $f:X\to Y$ be a set map. The map $f$ is said to be \textbf{smooth} if for every plot $p:U\to X$ of $X$ the composition 
$f\circ p$ is a plot of $Y$.
\end{definition}

The standard examples of diffeological spaces are smooth manifolds, with diffeology consisting of all usual smooth maps $\matR^k\supset U\to M$, for all $k\in\mathbb{N}$ and for all domains $U$ in $\matR^k$. 
A less standard example is, for instance, any Euclidean space $\matR^n$ with the so-called \emph{wire diffeology}, namely, the diffeology \emph{generated} (see the next paragraph) by the set
$C^{\infty}(\matR,\matR^n)$.

\paragraph{Generated diffeology} Let $X$ be a set, and let $A=\{U_i\to X\}_{i\in I}$ be a set of maps with values in $X$. The \textbf{diffeology generated by $A$} is the smallest, with respect to inclusion, 
diffeology on $X$ that contains $A$. It consists of all maps $f:V\to X$ such that there exists an open cover $\{V_j\}$ of $V$ such that $f$ restricted to each $V_j$ either is constant or factors through some
element $U_i\to X$ in $A$ via a smooth map $V_j\to U_i$. Note that this construction illustrates the abundance of diffeologies on a given set: we can build one starting from \emph{any} set of maps; in
particular, given any map $p:U\to X$, there is (usually more than one) diffeology containing it.

\paragraph{Fine diffeology and coarse diffeology} Given a set $X$, the set of all possible diffeologies on $X$ is partially ordered by inclusion, namely, a diffeology $\mathcal{D}$ on $X$ is said to be
\textbf{finer} than another diffeology $\mathcal{D}'$ if $\mathcal{D}\subset\mathcal{D}'$ (whereas $\mathcal{D}'$ is said to be \textbf{coarser} than $\mathcal{D}$). Among all diffeologies, there is the finest one, 
which turns out to be the natural \textbf{discrete diffeology} and which consists of all locally constant maps $U\to X$; and there is also the coarsest one, which consists of \emph{all} possible maps $U\to X$, for 
all $U\subseteq\matR^n$ and for all $n\in\mathbb{N}$. It is called \emph{the} \textbf{coarse diffeology}.

\paragraph{Subset diffeology} Let $X$ be a diffeological space, and let $Y\subset X$ be a subset of it. The \textbf{subset diffeology} on $Y$ is the \emph{coarsest} diffeology on $Y$ such that the inclusion 
map $Y\hookrightarrow X$ is smooth. What this means, specifically, is that $p:\matR^k\supset U\to Y$ is a plot for the subset diffeology on $Y$ if and only if its composition with the inclusion map is a plot of $X$.

\paragraph{Products and sums of diffeological spaces} Let $\{X_i\}_{i\in I}$ be a collection of diffeological spaces, where $I$ is a set of indices. The \textbf{product diffeology} $\mathcal{D}$ on the product
$\prod_{i\in I}X_i$ is the \emph{coarsest} diffeology such that for each index $i\in I$ the natural projection $\pi_i:\prod_{i\in I}X_i\to X_i$ is smooth. Note, in particular, that for the product of two spaces 
$X_1\times X_2$ every plot is locally of form $(p_1,p_2)$, where $p_i$ is a plot of $X_i$ for $i=1,2$.

If we now consider the disjoint union $\coprod_{i\in I}X_i$ of the spaces $X_i$, this can be endowed with the \textbf{sum diffeology}, which by definition is the \emph{finest} diffeology such that all the natural 
injections $X_i\hookrightarrow\coprod_{i\in I}X_i$ are smooth. The plots for this diffeology are those maps that are plots of one of the components of the sum.

\paragraph{Diffeological vector space} Let $V$ be a vector space over $\matR$. The \textbf{vector space diffeology} on $V$ is any diffeology of $V$ such that the addition and the scalar multiplication are 
smooth, that is,
$$[(u,v)\mapsto u+v]\in C^{\infty}(V\times V,V)\mbox{ and }[(\lambda,v)\mapsto\lambda v]\in C^{\infty}(\matR\times V,V),$$ where $V\times V$ and $\matR\times V$ are equipped with the product diffeology. A 
\textbf{diffeological vector space} over $\matR$ is any vector space $V$ over $\matR$ equipped with a vector space diffeology.

\paragraph{Fine diffeology on vector spaces} The \textbf{fine diffeology} on a vector space $\matR$ is the \emph{finest} vector space diffeology on it; endowed with such, $V$ is called a \textbf{fine vector space}. 
As an example, $\matR^n$ with the standard diffeology is a fine vector space.

The fine diffeology admits a more or less explicit description of the following form: its plots are maps $f:U\to V$ such that for all $x_0\in U$ there exist an open neighbourhood $U_0$ of $x_0$, a family of smooth 
maps $\lambda_{\alpha}:U_0\to\matR$, and a family of vectors $v_{\alpha}\in V$, both indexed by the same finite set of indices $A$, such that $f|_{U_0}$ sends each $x\in U_0$ into
$\sum_{\alpha\in A}\lambda_{\alpha}(x)v_{\alpha}$:
$$f(x)=\sum_{\alpha\in A}\lambda_{\alpha}(x)v_{\alpha}\mbox{ for }x\in U_0.$$ A finite family $(\lambda_{\alpha},v_{\alpha})_{\alpha\in A}$, with $\lambda\in C^{\infty}(U_0,\matR)$ and $v_{\alpha}\in V$, defined
on some domain $U_0$ and satisfying the condition just stated, is called a \emph{local family} for the plot $f$.

Fine vector spaces possess the following property (\cite{iglesiasBook}, 3.9), which in general not true: if $V$ is a fine diffeological vector space, and $W$ is any other diffeological vector space, then every 
linear map $V\to W$ is smooth, \emph{i.e.}, $L^{\infty}(V,W)=L(V,W)$.

\paragraph{Smooth scalar products} The existing notion of a \textbf{Euclidean diffeological vector space} does not differ much from the usual notion of the Euclidean vector space. A diffeological space $V$ is 
Euclidean if it is endowed with a scalar product that is smooth with respect to the diffeology of $V$ and the standard diffeology of $\matR$; that is, if there is a fixed map $\langle , \rangle:V\times V\to\matR$ that 
has the usual properties of bilinearity, symmetricity, and definite-positiveness and that is smooth with respect to the diffeological product structure on $V\times V$ and the standard diffeology on $\matR$. As has 
already been mentioned in the Introduction, for finite-dimensional spaces this implies that the space in question is just the usual $\matR^n$, for appropriate $n$ (see \cite{iglesiasBook}). For reasons of 
completeness we recall the precise statement and give a detailed proof of this fact.\footnote{Special thanks go to Patrick Iglesias-Zemmour and Yael Karshon who made me first think of it; I would have missed 
it otherwise.}

\begin{proposition}\label{smooth:prod-smooth:plot}
Let $V$ be $\matR^n$ endowed with a vector space diffeology $\mathcal{D}$ such that there exists a smooth scalar product. Then every plot $p$ of $\mathcal{D}$ is a smooth map in the usual sense.
\end{proposition}

\begin{proof}
Let $A$ an $n\times n$ non-degenerate symmetric matrix such that the associated bilinear form on $V$ is smooth with respect to the diffeology $\mathcal{D}$, and let $\{v_1,\ldots,v_n\}$ be its eigenvector 
basis. Let $\lambda_i$ be the eigenvalue relative to the eigenvector $v_i$.

Let $p:U\to V$ be a plot of $\mathcal{D}$; we wish to show that it is smooth as a map $U\to\matR^n$. Recall that $\langle\cdot|\cdot\rangle_A$ being smooth implies that for any two plots $p_1,p_2:U\to V$ the 
composition $\langle\cdot|\cdot\rangle_A\circ(p_1,p_2)$ is smooth as a map $U\to\matR$. Let $c_i:U\to V$ be the constant map $c_i(x)=v_i$; this is of course a plot of $\mathcal{D}$. Set $p_1=p$ and
$p_2=c_i$; then the above composition map writes as $\lambda_i\langle p(x)|v_i\rangle$, where $\langle\cdot|\cdot\rangle$ is the canonical scalar product on $\matR^n$.

Since $A$ is non-degenerate, all $\lambda_i$ are non-zero; this implies that each function $\langle p(x)|v_i\rangle$ is a smooth map. And since $v_1,\ldots,v_n$ form a basis of $V$, this implies that for any 
$v\in V$ the function $\langle p(x)|v\rangle$ is a smooth one. In particular, this is true for any $e_j$ in the canonical basis of $\matR^n$; and in the case $v=e_j$ the scalar product $\langle p(x)|e_j\rangle$ is 
just the $j$-th component of $p(x)$. Thus, we obtain that all the components of $p$ are smooth functions, therefore $p$ is a smooth map.
\end{proof}

\section{The degeneracy of smooth forms on non-standard spaces}

As mentioned in the introduction, pretty much all finite-dimensional diffeological vector spaces do not have smooth non-degenerate bilinear forms; there is only one for each dimension that does, and that is the 
standard one. For exposition purposes, we start with a detailed illustration of how this happens, via an example (this part is however easily deduced from the proof of Proposition \ref{smooth:prod-smooth:plot} 
above). We then consider the degree of degeneracy of smooth bilinear forms on a given vector space (what does it mean, for a given $V$, to be the least-degenerate bilinear form on it?), going on to the 
question of subspaces which do, or do not, split off as smooth direct summands, and finally considering the diffeological dual.

\subsection{Pseudo-metrics}

A \emph{pseudo-metric} is, roughly speaking, the best possible substitute for the notion of a smooth scalar product in the case of a finite-dimensional vector space whose diffeology is not the standard one.

\paragraph{The first example} The following example, already considered in \cite{multilinear}, is presented for illustrative purposes.

\begin{example}
Let $V=\matR^n$, and let $v_0\in V$ be any non-zero vector. Let $p:\matR\to V$ be defined as $p(x)=|x|v_0$; let $\mathcal{D}$ be any vector space diffeology on $V$ that contains $p$ as a plot.\footnote{Such 
diffeology does certainly exist; for instance, the coarse diffeology would do.} Suppose that $A$ is a symmetric $n\times n$ matrix, and assume that the bilinear form $\langle v|w\rangle_A=v^tAw$ associated to 
$A$ is smooth with respect to $\mathcal{D}$ and the standard diffeology on $\matR$. We claim that $A$ is degenerate.

Indeed, $\langle v|w\rangle_A$ being smooth implies, in particular, that for any two plots $p_1,p_2:\matR\to V$ of $V$ the composition map $\langle \cdot|\cdot\rangle_A\circ (p_1,p_2):\matR\to\matR$ is smooth 
in the usual sense; this map acts as $\matR\ni x\mapsto(p_1(x))^tAp_2(x)$. Let $w\in V$ be an arbitrary vector; denote by $c_w:\matR\to V$ the constant map that sends everything to $w$, $c_w(x)=w$ for all 
$x\in\matR$. Such a map is a plot for any diffeology on $V$. But then $(\langle \cdot|\cdot\rangle_A\circ (p,c_w))(x)=|x|v_0^tAw$; the only way for this to be smooth is to have $v_0^tAw=0$, and since there was 
no assumption on $w$, this implies that $\langle v_0|\cdot\rangle_A$ is identically zero on the whole of $V$, \emph{i.e.}, that $A$ is degenerate. In other words, $V$ does not admit a smooth scalar product.
\end{example}

Note that in the above example we could have taken $p(x)=f(x)v_0$ with $f(x)$ \emph{any} function $\matR\to\matR$ that is not differentiable in at least one point; this suggests that there are numerous 
diffeological vector spaces that do not admit diffeologically smooth scalar products. In fact, Proposition \ref{smooth:prod-smooth:plot} shows that the phenomenon is much more general: in order to have a 
smooth scalar product, we must ensure that all plots are smooth maps.

\paragraph{The signature of a smooth bilinear form} Suppose that $V$ is a finite-dimensional diffeological vector space; let $A$ be a symmetric $n\times n$ matrix (with $n=\dim V$) such that the associated 
bilinear form $\langle\cdot|\cdot\rangle_A$ on $V$ is smooth. Let $(\lambda_+,\lambda_-,\lambda_0)$ be the signature of this form; recall that $V^*$ stands for the \emph{diffeological dual} of $V$ (see 
\cite{vincent} and \cite{wu}), \emph{i.e.} the set of all smooth linear functionals on $V$.

\begin{lemma}\label{eigenvalue0:smooth:form:lem}
Let $V^*$ be the diffeological dual of $V$. Then
$$\lambda_0\geq n-\dim(V^*).$$
\end{lemma}

\begin{proof}
Choose a basis $\{v_1,\ldots,v_n\}$ of eigenvectors of $A$ such that the last $\lambda_0$ vectors are those relative to the eigenvalue $0$. For $i=1,\ldots,n-\lambda_0$ let $v^i\in V^*$ be the dual function, 
that is, $v^i(w)=\langle v_i|w \rangle_A$ (it is obviously smooth since $\langle\cdot|\cdot\rangle_A$ is smooth, so is an element of $V^*$). It remains to notice that by standard reasoning the elements 
$v^1,\ldots,v^{n-\lambda_0}$ are linearly independent (they belong to non-zero eigenvalues), so $n-\lambda_0\leq\dim(V^*)$, hence the conclusion.
\end{proof}

Notice that we can always find a smooth bilinear (symmetric) form for which we have $\lambda_0=n-\dim(V^*)$. It suffices to take any basis $\{f_1,\cdots,f_k\}$ of $V^*$ (with $k=\dim(V^*)$) and consider 
$\sum_i f_i\otimes f_i$; this is obviously a symmetric bilinear form on $V$, and it being smooth follows from Theorem 2.3.5 of \cite{vincent}. By an obvious argument, we can see that all the eigenvalues are 
non-negative; indeed, $(\sum_i f_i\otimes f_i)(v\otimes v)=\sum_i(f_i(v))^2\geq 0$. Finally, the multiplicity of the eigenvalue $0$ is precisely $n-k$, since $f_1,\cdots,f_k$ form a basis of $V^*$.

Thus, the following definition makes sense.

\begin{definition}
Let $V$ be a diffeological vector space of finite dimension $n$. A \textbf{diffeological pseudo-metric} on $V$ is a smooth bilinear symmetric form on $V$ such that the eigenvalue $0$ has multiplicity
$n-\dim(V^*)$, and all the other eigenvalues are positive.
\end{definition}

A pseudo-metric is the best notion of a smooth\footnote{Which means a bilinear form that takes values in $\matR$ with standard diffeology and is smooth with respect to the latter.} metric that can exist for an 
arbitrary finite-dimensional diffeological vector space; note that if a given space is standard, \emph{i.e.} is of the only type which admits a smooth metric, then obviously a pseudo-metric is a usual Euclidean 
metric. Of course, one can look for true metrics which are not smooth but are at least piecewise-differentiable, or simply continuous, something that we will do in the section that follows; meanwhile, we explore 
whatever applications pseudo-metrics might have.

\begin{remark}
While it is more usual to use scalar products for various constructions of (multi)linear algebra, we do note that a number of them hold in part for symmetric bilinear forms. A specific example would be a Clifford 
algebra, which is defined for a vector space endowed with a symmetric bilinear form that \emph{a priori} does not have to be a scalar product (an extension of this notion into the diffeological context was 
considered in \cite{clifford}). Obviously, if it is not then one cannot speak of unitary action on the exterior algebra, for example; yet, various constructions hold.
\end{remark}

\subsection{Smooth direct sums and pseudo-metrics}

One easily observed fact that distinguishes diffeological vector spaces from their usual linear counterpart is that, in the diffeological case, there are two types of direct sum: one that is smooth (in the sense 
described just below) and the other that is not.

\paragraph{Smooth splittings as direct sums} In general, if we have a diffeological vector space $V$ (of finite dimension) that, as a usual vector space, decomposes into a direct sum $V=V_1\oplus V_2$,
then \emph{a priori} the following two diffeologies are different:
\begin{itemize}
\item the given diffeology of $V$, and
\item the vector space sum diffeology on $V_1\oplus V_2$ obtained from the subset diffeologies on $V_1$ and $V_2$.\footnote{The vector space sum diffeology on a direct sum of diffeological vector spaces 
is the finest vector space diffeology that contains the sum diffeology; it is formed of all maps that locally are (formal) sums of plots of the summands.}
\end{itemize}

\begin{example}
Let us consider an example where the two might actually be different. Take $V=\matR^3$ endowed with the vector space diffeology that is generated by the map $p:\matR\to V$ acting by $p(x)=|x|(e_2+e_3)$. 
Set $V_1=\mbox{Span}(e_1,e_2)$ and $V_2=\mbox{Span}(e_3)$. Observe that the subset diffeology on each of these subspaces is just the standard diffeology of $\matR^2$ and $\matR$ respectively.
\footnote{Let us prove this claim. Let $q:U\to V_1$ be a plot for the subset diffeology on $V_1$; let $q_i:U\to\matR$ be the projection on $\mbox{Span}(e_i)\leq V_1$ for $i=1,2$. It suffices to show that $q_i$ 
must necessarily be a smooth map (in the usual sense). Note that by definition of the subset diffeology $Q(u)=(q_1(u),q_2(u),0)$ must be a plot for the diffeology of $V$. By definition of the latter diffeology in 
the neighbourhood of zero it writes as $Q(u)=\sum_{i=1}^3 f_i(u)e_i+\sum_{j=1}^k h_j(u)\cdot(p\circ F_j)(u)$, where $f_1,f_2,f_3$, $h_1,\ldots,h_k$, and $F_1,\ldots,F_k$ are the usual smooth maps 
$U\to\matR$; therefore up to smooth summands $Q(u)$ writes as $(0,\sum_j h_j(u)|F_j(u)|,\sum_j h_j(u)|F_j(u)|)$, from which we conclude that $\sum_j h_j(u)|F_j(u)|\equiv 0$ in the neighbourhood of $u=0$. 
This implies that $q_1,q_2$ are the usual smooth maps, and the claim is proved for $V_1$; the reasoning is analogous in the case of $V_2$.} This means that the vector space sum diffeology on their direct 
sum is again the standard one, which the diffeology on $V$ is not; indeed, the projection on the second coordinate is smooth (in the usual sense) for the standard diffeology, which does not happen for the 
diffeology of $V$.
\end{example}

\paragraph{Standard subspaces that split off as direct summands} The example given in the previous paragraph easily leads to some more general observations, such as the following.

\begin{lemma}\label{proj:smooth:lem}
Let $V$ be a finite-dimensional diffeological vector space, let $V_1\leq V$ be a subspace such that the subset diffeology is standard, and let $V=V_1\oplus V_2$ be its decomposition into a direct sum for some 
subspaces $V_1,V_2\leq V$. This decomposition is smooth if and only if there exists a basis $\{v_1,\cdots,v_k\}$ of $V_1$ such that the projection\footnote{By this we mean the map $V\to\matR$ acting by 
$\pi_i(\sum_i\alpha_iv_i+w)=\alpha_i$, where $w\in V_2$.} on each vector $v_i$ is a smooth linear functional on $V$. 
\end{lemma}

\begin{proof}
This is a consequence of Theorem 3.16 of \cite{wu}. Indeed, suppose that $V=V_1\oplus V_2$ is a smooth decomposition. The projection on $V_1$ is then smooth by definition, and by assumption that $V_1$ 
is standard, so is the projection of $V_1$ onto any $v\in V_1$. The projection of $V$ onto the vector $v$, being the composition of the two projections, is therefore smooth as well.

Suppose now that there exists a basis $\{v_1,\cdots,v_k\}$ of $V_1$ as in the statement. Since $V_1$ is standard and all projections $V$ on $v_i$ are smooth, the projection of $V$ onto $V_1$ is smooth as 
well, and it suffices to use the equivalence of conditions (1) and (3) of Theorem 3.16 of \cite{wu} to obtain the desired claim.
\end{proof}

\paragraph{The maximal standard subspace relative to a pseudo-metric} As follows from the above discussion, a pseudo-metric exists on any diffeological vector space (and we cannot do any better). On the
other hand, given a pseudo-metric $\langle\cdot|\cdot\rangle_A$ on a fixed diffeological vector space $(V,\mathcal{D})$ of finite dimension, the subspace $V_0$ of $V$ spanned by all the eigenvectors
belonging to the positive eigenvalues of $\langle\cdot|\cdot\rangle_A$ and considered with the subset diffeology $\mathcal{D}_0$ induced by $\mathcal{D}$ is endowed, via the restriction of 
$\langle\cdot|\cdot\rangle_A$, with a true metric; it follows then from \cite{iglesiasBook}, Ex. 70, p.74 that $(V_0,\mathcal{D}_0)$ is the standard Euclidean space of dimension $n-\lambda_0=n-\dim(V^*)$.

\begin{lemma}\label{V_0:smooth:summand:lem}
The subspace $V_0$ splits off as a smooth direct summand.
\end{lemma}

\begin{proof}
Denote by $V_1$ the subspace generated by all the eigenvectors relative to the eigenvalue $0$. Obviously, $V$ splits as the direct sum $V=V_0\oplus V_1$ in the usual sense; we need to check that this
decomposition is smooth, that is, that the diffeology on $V$ coincides with the direct sum diffeology on $V_0\oplus V_1$. Now, by Lemma \ref{proj:smooth:lem}, it suffices to check that the projection on each 
eigenvector in some basis of $V_0$ is a smooth linear functional on the whole of $V$.

The matrix $A$ being symmetric, there exists an orthonormal (with respect to the usual scalar product on $\matR^n$ underlying $V$) basis $\{v_1,\cdots,v_n\}$ of $V$ composed of eigenvectors of $A$,
where we can obviously assume that the eigenvectors relative to all the positive eigenvalues (and so forming a basis of $V_0$) are the first $k$ vectors, $v_1,\ldots,v_k$. Let $v_i$ be one of these, relative to 
the eigenvalue $\lambda_i$, let $\pi:V\to\matR$ be the corresponding projection, and let $p:U\to V$ be any plot of $V$; we need to show that $\pi\circ p:U\to\matR$ is smooth in the usual sense. Write 
$p(x)=\sum_{j=1}^n p_j(x)v_j$; recall that $p_i(x)=\pi\circ p$.

Since $A$ defines a smooth bilinear form on $V$, and all the constant maps are plots, the assignment $x\mapsto\langle v_i|p(x) \rangle_A$ defines a map $U\to\matR$ which is smooth in the usual sense. 
However, we have $\langle v|p(x)\rangle_A=\sum_{j=1}^n\lambda_jp_j(x)\langle v_i|v_j\rangle=\lambda_i||v_i||^2(\pi\circ p)(x)$, where $||v_i||$ is the usual Euclidean norm of $v_i$. Since neither it nor 
$\lambda_i$ are zero, we conclude that $\pi\circ p$ is smooth, as wanted.
\end{proof}

Furthermore, the following is true.

\begin{proposition}\label{V_0:maximal:prop}
The subspace $V_0$ is a maximal subspace of $V$ with the subset diffeology that is standard and which splits off as a smooth direct summand.
\end{proposition}

\begin{proof}
Suppose that $V_1>V_0$ is a bigger subspace of $V$ such that its subset diffeology is standard and $V_1$ also splits off as a smooth direct summand; let $V=V_1\oplus V_1'$ be the corresponding smooth 
direct sum decomposition. Define a new bilinear bilinear form $\varphi$ on $V$ by setting it to be the zero form on $V_1'$, a scalar product on $V_1$, and $V_1$ and $V_1'$ to be orthogonal. Since by 
assumption the decomposition is smooth and $V_1$ has the standard diffeology, $\varphi$ is smooth, and by construction it is symmetric and positive semi-definite. Furthermore, it has rank 
$\dim(V_1)>\dim(V_0)$, which contradicts Lemma \ref{eigenvalue0:smooth:form:lem}, and we obtain the claim.
\end{proof}

In other words, a pseudo-metric on $V$ allows to extract, in a sense, from $V$ its biggest standard part, on which the diffeology includes essentially the usual smooth maps. Let us be more precise.

\begin{corollary}
Let $V=W_0\oplus W_1$ be a smooth decomposition such that $\dim(W_0)\geq\dim(V^*)$, and the subset diffeology on $W_0$ is standard. Then there exists a pseudo-metric $\langle\cdot|\cdot\rangle_A$ on 
$V$ such that $W_0$ is the space generated by the eigenvectors of $A$ relative to the positive eigenvalues, and $W_1$ is the space generated by the eigenvectors relative to the eigenvalue $0$.
\end{corollary}

In particular, $\dim(W_0)=\dim(V^*)$. To summarize, all subspaces of $V$ with standard sub-diffeology that in addition split off as smooth direct summands are associated to some pseudo-metric, provided 
they have a sufficiently large dimension. We can actually show more: there is only one of them.

\begin{proposition}
The subspace $V_0$ is an invariant of the finite-dimensional diffeological space $V$, \emph{i.e.}, it does not depend on the choice of a pseudo-metric $\langle\cdot|\cdot\rangle_A$.
\end{proposition}

\begin{proof}
Let $W_0\neq V_0$ be another subspace of $V$ of dimension $k=\dim(V^*)$ such that the subset diffeology on $W_0$ is standard and $W_0$ splits off as a smooth summand, $V=W_0\oplus W_1$. Let
$w_0\in W_0\setminus V_0$; it writes uniquely as $w_0=v_0+v_1$. Note that by Lemma \ref{proj:smooth:lem} the projections on both $w_0$ and $v_0$ are smooth linear functionals; therefore so is the
projection on $v_1$. That would imply that $\mbox{Span}(V_0,v_1)$ has standard diffeology and splits off as a smooth summand; since its dimension is strictly greater than $\dim(V_0)=\dim(V^*)$, this
is a contradiction.
\end{proof}

\subsection{The metric on $V^*$ induced by a pseudo-metric on $V$}

Here we consider a finite-dimensional diffeological vector space $V$ (we assume that its underlying vector space is identified with $\matR^n$), endowed with a pseudo-metric $\langle\cdot|\cdot\rangle_A$ 
given by a symmetric matrix $A$. As for usual vector spaces, the pseudo-metric induces a symmetric bilinear form on the diffeological dual $V^*$ of $V$; in this section we deal with this induced form.

\paragraph{The functional diffeology on $V^*$} Recall \cite{wu} that in general, the diffeological dual $V^*$ of $V$ is \emph{not} even isomorphic to $V$, so not diffeomorphic to it; this occurs already in the 
finite-dimensional case (in addition to \cite{wu}, see Example 3.1 in \cite{multilinear}). Thus, it might well occur that while the diffeology on $V$ is not standard, the corresponding functional diffeology on its 
dual is. The following statement makes it precise.

\begin{theorem}\label{dual:diff-gy:standard:thm}
Let $V$ be a finite-dimensional diffeological vector space, and let $V^*$ be its diffeological dual. Then the functional diffeology on $V^*$ is standard (\emph{i.e.}, $V^*$ is diffeomorphic to some $\matR^k$ with 
standard diffeology, for the appropriate $k$).
\end{theorem}

\begin{proof}
This is an easy consequence of Remark 3.6 and Property 8 (Section 5) of \cite{wu}. Indeed, let $V_0$ be the maximal standard subspace of $V$ that splits off smoothly, and let $V=V_0\oplus V_1$ be the 
corresponding smooth decomposition. It follows from Remark 3.6 of \cite{wu} that 
$$V^*=(V_0\oplus V_1)^*\cong V_0^*\times V_1^*=V_0^*.$$ By Property (8) in \cite{wu}, since $V_0$ is fine and finite-dimensional, so is $V_0^*$; in particular, it is a standard space.
\end{proof}

\paragraph{The subspace $V_0$ and the diffeological dual} The proofs of Lemma \ref{V_0:smooth:summand:lem} and of Theorem \ref{dual:diff-gy:standard:thm} suggest that each pseudo-metric on $V$ gives 
a natural identification between the corresponding $V_0$ and $V^*$ (that the two spaces are \emph{a priori} diffeomorphic is now obvious, since they are both standard spaces of the same dimension). More 
precisely, the following is true.

\begin{theorem}\label{diffeo:V_0:dual:thm}
Let $V$ be a finite-dimensional diffeological vector space, and let $\langle\cdot|\cdot\rangle_A$ be a pseudo-metric on $V$. The restriction on the subspace $V_0$ of the induced map $\Psi:V\to V^*$ 
(\emph{i.e.}, the map given by $v\mapsto[w\mapsto\langle w|v\rangle_A]$) is a diffeomorphism.
\end{theorem}

\begin{proof}
As has been already observed in the proof of Theorem \ref{dual:diff-gy:standard:thm}, the subspace $V_0$ admits an orthonormal (with respect to the canonical scalar product on $\matR^n$ underlying $V$) 
basis $\{v_1,\ldots,v_k\}$ composed of eigenvectors of $A$; furthermore, if $f_i=\Psi(v_i)$ for $i=1,\ldots,k$ then $f_i$ form a basis of $V^*$. In particular, the restriction of $\Psi$ to $V_0$ is a bijection with 
$V^*$, hence an isomorphism.

Now, applying the reasoning from the proof of Lemma \ref{V_0:smooth:summand:lem} we find that each plot $p$ of $V_0$ is of form $p(u)=p_1(u)v_1+\cdots+p_k(u)v_k$ for some smooth real-valued maps 
$p_1,\cdots,p_k$. Hence the composition $\Psi\circ p$ writes as $(\Psi\circ p)(u)=p_1(u)f_1+\cdots+p_k(u)f_k$, which is obviously a plot of $V^*$ (the maps $p_1,\cdots,p_k$ being smooth, this is a plot for any 
vector space diffeology on $V^*$), therefore $\Psi$ is smooth. \emph{Vice versa}, arguing as in the proof of Theorem \ref{dual:diff-gy:standard:thm}, every plot $q$ of $V^*$ writes as 
$q(u)=q_1(u)f_1+\cdots+q_k(u)f_k$ for some smooth real-valued maps $q_1,\cdots,q_k$. Thus, the composition $\Psi^{-1}\circ q$ is of form $(\Psi^{-1}\circ q)(u)=q_1(u)v_1+\cdots+q_k(u)v_k$, which is a plot 
of $V_0$. We conclude that $\Psi$ is smooth with smooth inverse, hence the conclusion.
\end{proof}

\paragraph{The induced metric on $V^*$} It is now easy to see that any pseudo-metric on $V$ induces a true metric on $V^*$, via the smooth surjection $\Psi$ described in Theorem \ref{diffeo:V_0:dual:thm}. 
Let us describe this precisely; notice first of all that if $f\in V^*$ then there is a single element $v_0\in\Psi^{-1}(f)$ such that $v_0\in V_0$. Furthermore, every other element in $\Psi^{-1}(f)$ is of form $v_0+v_1$, 
where $v_1\in V^{\perp}$, the orthogonal of the whole $V$ with respect to the pseudo-metric $\langle\cdot|\cdot\rangle_A$ (equivalently, the subspace $V_1$ spanned by all the eigenvectors relative to the
eigenvalue $0$). Let now $g\in V^*$ be another element of the dual $V^*$, and let $w_0+w_1\in\Psi^{-1}(g)$. Then obviously, 
$\langle v_0+v_1|w_0+w_1\rangle_A=\langle v_0|w_0\rangle_A=\langle(\Psi|_{V_0})^{-1}(f)|(\Psi|_{V_0})^{-1}(g)\rangle_A$. The first equality implies that the pushforward of the pseudo-metric
$\langle\cdot|\cdot\rangle_A$ to $V^*$ is well-defined (it could have been replaced by a reference to \cite{wu}, Proposition 3.12 (1), of which it is a partial instance); the second equality means that this 
pushforward is a true metric. We summarize this discussion in the following statement.

\begin{corollary}
Any pseudo-metric $\langle\cdot|\cdot\rangle_A$ on $V$ induces a true metric on the diffeological dual $V^*$ of $V$, via the natural pairing that assigns to each $v\in V$ the smooth linear functional
$\langle\cdot|v\rangle_A$.
\end{corollary}

Using the reasoning employed so far, we can also establish the inverse of this statement, namely, that \emph{every smooth metric (scalar product) on $V^*$ is induced by a pseudo-metric on $V$.} Indeed, let 
$\langle\cdot|\cdot\rangle_B$ be smooth metric on $V^*$, given by a $k\times k$ matrix $B$. Let $\{f_1,\cdots,f_k\}$ be an orthonormal basis of $V^*$ composed of eigenvectors of $B$.\footnote{Note that the 
basis $\{f_1,\cdots,f_k\}$ generates the standard diffeology on $V^*$, in the sense that each plot $q:U\to V^*$ is locally of form $q(u)=q_1(u)f_1+\cdots+q_k(u)f_k$ for some smooth functions 
$q_1,\cdots,q_k:U\to\matR$. Indeed, each plot does write in this manner, simply by virtue of $\{f_1,\cdots,f_k\}$ being a basis; furthermore, $\langle\cdot|\cdot\rangle_B$ being smooth, and each constant map 
being a plot, implies that $\langle q(u)|f_i\rangle_B=q_i(u)$ is smooth in the usual sense as a map $U\to\matR$.} As has already been noted, the form $\sum_{i=1}^kf_i\otimes f_i$ gives a certain pseudo-metric
$\langle\cdot|\cdot\rangle_A$ on $V$; it remains to check that this pseudo-metric does induce $\langle\cdot|\cdot\rangle_B$ on $V^*$, and this follows from the fact that $\{f_1,\cdots,f_k\}$ is an orthonormal 
basis.

\subsection{A simple example of a pseudo-metric}

Let us return to an above example, considering $V=\matR^3$ endowed with the vector space diffeology generated by the plot $p:\matR\to V$ given by $p(x)=|x|(e_2+e_3)$. Observe first that the diffeological 
dual of $V$ is generated by maps $e^1$ and $e^2-e^3$ (with $\{e^1,e^2,e^3\}$ being the canonical basis of the usual dual of $\matR^3$). In particular, $\dim(V^*)=2$.

It is also easy to see that any smooth symmetric bilinear form on $V$ is given by a matrix of form
${\small{\left(\begin{array}{ccc} 
c & a & -a \\
a & b & -b \\
-a & -b & b \end{array}\right)}}$ for some $a,b,c\in\matR$ (the vector $(0,|x|,|x|)$ must be orthogonal to any other vector of $\matR^3$ with respect to the associated bilinear form). For this matrix to define a
pseudo-metric, we must have $b,c>0$ and $a^2<bc$.

To give a specific example, we can take $a=1$, $b=c=2$, obtaining
$A={\small{\left(\begin{array}{ccc} 
2 & 1 & -1 \\
1 & 2 & -2 \\
-1 & -2 & 2 \end{array}\right)}}$. The two positive eigenvalues of $A$ are $\lambda_1=3+\sqrt3$ and $\lambda_2=3-\sqrt3$; the corresponding (non unitary) eigenvectors are 
$v_1'={\small{\left(\begin{array}{c} 
3+\sqrt3 \\ 
3+2\sqrt3 \\ 
-3-2\sqrt3 \end{array}\right)}}$ and
$v_2'={\small{\left(\begin{array}{c} 
3-\sqrt3 \\ 
3-2\sqrt3 \\ 
2\sqrt3-3 \end{array}\right)}}$. We can however easily find that the subspace $V_0$ can be (better) described as $V_0=\mbox{Span}(e_1,e_2-e_3)$; the restriction of $\langle\cdot|\cdot\rangle_A$ to $V_0$, 
with respect to the basis $\{e_1,e_2-e_3\}$ has the matrix 
${\small{\left(\begin{array}{cc} 2 & 2 \\ 2 & 8 \end{array}\right)}}$. We also note that the subspace $V_1$ (generated by eigenvectors relative to $0$) is 
$\mbox{Span}({\small{\left(\begin{array}{c} 0 \\ 1 \\ 1 \end{array}\right)}})$.

Finally, let us calculate the matrix of the induced metric on $V^*$ with respect to its basis $\{e^1,e^2-e^3\}$. For this, we need to find their pre-images with respect to $\Psi|_{V_0}$ (we just write $\Psi$ for 
brevity). By an easy calculation $\Psi^{-1}(e^1)=\frac23 e_1-\frac16(e_2-e_3)$ and $\Psi^{-1}(e^2-e^3)=-\frac13 e_1+\frac13(e_2-e_3)$. Calculating their pairwise products with respect to the pseudo-metric
$\langle\cdot|\cdot\rangle_A$, we find the matrix
$\frac19{\small{\left(\begin{array}{cc} 6 & 5 \\ 5 & 6 \end{array}\right)}}$.

\vspace{1cm}

\noindent University of Pisa \\
Department of Mathematics \\
Via F. Buonarroti 1C\\
56127 PISA -- Italy\\
\ \\
ekaterina.pervova@unipi.it\\


\begin{thebibliography}{99}
\bibitem{iglesiasBook}
\textsc{P. Iglesias-Zemmour},  \textit{Diffeology}, Mathematical Surveys and Monographs, 185, AMS, Providence, 2013.
\bibitem{multilinear}
\textsc{E. Pervova}, \textit{Multilinear algebra in the context of diffeology}, arXiv:1504.08186v2.
\bibitem{clifford}
\textsc{E. Pervova}, \textit{Diffeological Clifford algebras}, arXiv:1505.06894v2.
\bibitem{So1}
\textsc{J.M. Souriau}, \textit{Groups diff\'erentiels}, Differential geometrical methods in mathematical physics (Proc. Conf., Aix-en-Provence/Salamanca, 1979), Lecture Notes in Mathematics, 836, 
Springer, (1980), pp. 91-128.
\bibitem{So2}
\textsc{J.M. Souriau}, \textit{Groups diff\'erentiels de physique math\'ematique}, South Rhone seminar on geometry, II (Lyon, 1984), Ast\'erisque 1985, Num\'ero Hors S\'erie, pp. 341-399.
\bibitem{vincent}
\textsc{M. Vincent}, \textit{Diffeological differential geometry}, Master Thesis, University of Copenhagen, 2008, available at 
{\tt http://www.math.ku.dk/english/research/top/paststudents/\\martinvincent.msthesis.pdf}
\bibitem{wu}
\textsc{E. Wu}, \textit{Homological algebra for diffeological vector spaces}, Homology, Homotopy $\&$ Applications (1) \textbf{17} (2015), pp. 339-376.
\end{thebibliography}
\end{document}